\documentclass[12pt]{amsart}
\usepackage{amsmath,amssymb,fullpage}
\usepackage[active]{srcltx} 

\newtheorem{theorem}{Theorem}[section]

\newtheorem{corollary}[theorem]{Corollary}

\newtheorem{proposition}[theorem]{Proposition}

\numberwithin{equation}{section}

\newcommand{\R}{\mathbb{R}}
\newcommand{\N}{\mathbb{N}}
\newcommand{\Z}{\mathbb{Z}}
\newcommand{\Q}{\mathbb{Q}}
\newcommand{\C}{\mathbb{C}}

\newcommand{\dis}{\displaystyle}
\newcommand{\supp}{\textup{supp}}
\newcommand{\tr}{\textup{tr}}
\newcommand{\eps}{\varepsilon}

\begin{document}

\title[Distribution of algebraic numbers]
{Distribution of algebraic numbers}%
\author{Igor E. Pritsker}%

\thanks{Research was partially supported by the National Security
Agency, and by the Alexander von Humboldt Foundation.}

\address{Department of Mathematics, Oklahoma State University, Stillwater, OK 74078, U.S.A.}%
\email{igor@math.okstate.edu}




\begin{abstract}

Schur studied limits of the arithmetic means $A_n$ of zeros for polynomials of degree $n$ with integer coefficients and simple zeros in the closed unit disk. If the leading coefficients are bounded, Schur proved that $\limsup_{n\to\infty} |A_n| \le 1-\sqrt{e}/2.$ We show that $A_n \to 0$, and estimate the rate of convergence by generalizing the Erd\H{o}s-Tur\'an theorem on the distribution of zeros. As an application, we show that integer polynomials have some unexpected restrictions of growth on the unit disk.

Schur also studied problems on means of algebraic numbers on the real line. When all conjugate algebraic numbers are positive, the problem of finding the sharp lower bound for $\liminf_{n\to\infty} A_n$ was developed further by Siegel and others. We provide a solution of this problem for algebraic numbers equidistributed in subsets of the real line.

Potential theoretic methods allow us to consider distribution of algebraic numbers in or near general sets in the complex plane. We introduce the generalized Mahler measure, and use it to characterize asymptotic equidistribution of algebraic numbers in arbitrary compact sets of capacity one. The quantitative aspects of this equidistribution are also analyzed in terms of the generalized Mahler measure.

\end{abstract}

\maketitle


\section{Schur's problems on means of algebraic numbers}

Let $E$ be a subset of the complex plane $\C.$ Consider the set of polynomials $\Z_n(E)$ of the exact degree $n$ with integer coefficients and all zeros in $E$. We denote the subset of $\Z_n(E)$ with simple zeros by $\Z_n^s(E)$. Given $M>0$, we write $P_n=a_nz^n + \ldots\in\Z_n^s(E,M)$ if $|a_n|\le M$ and $P_n\in\Z_n^s(E)$ (respectively $P_n\in\Z_n(E,M)$ if $|a_n|\le M$ and $P_n\in\Z_n(E)$). Schur \cite{Sch}, \S 4-8, studied the limit behavior of the arithmetic means of zeros for polynomials from $\Z_n^s(E,M)$ as $n\to\infty,$ where $M>0$ is an arbitrary fixed number. His results may be summarized in the following statements. Let $\R_+:=[0,\infty),$ where $\R$ is the real line.

\noindent{\bf Theorem A} (Schur \cite{Sch}, Satz IX) {\em Given a polynomial $P_n(z)=a_n\prod_{k=1}^n (z-\alpha_{k,n})$, define the arithmetic mean of squares of its zeros by $S_n:=\sum_{k=1}^n \alpha_{k,n}^2/n.$ If $P_n\in\Z_n^s(\R,M)$ is any sequence of polynomials with degrees $n\to\infty$, then}
\begin{align} \label{1.1}
\liminf_{n\to\infty} S_n \ge \sqrt{e} > 1.6487.
\end{align}

\noindent{\bf Theorem B} (Schur \cite{Sch}, Satz XI) {\em For a polynomial $P_n(z)=a_n\prod_{k=1}^n (z-\alpha_{k,n})$, define the arithmetic mean of its zeros by $A_n:=\sum_{k=1}^n \alpha_{k,n}/n.$ If $P_n\in\Z_n^s(\R_+,M)$ is any sequence of polynomials with degrees $n\to\infty$, then}
\begin{align} \label{1.2}
\liminf_{n\to\infty} A_n \ge \sqrt{e} > 1.6487.
\end{align}
It is clear that Theorems A and B are connected by the transformation $w=z^2.$ Let $D:=\{z:|z|\le 1\}$ be the closed unit disk.

\medskip
\noindent{\bf Theorem C} (Schur \cite{Sch}, Satz XIII) {\em If $P_n\in\Z_n^s(D,M)$ is any sequence of polynomials with degrees $n\to\infty$, then}
\begin{align} \label{1.3}
\limsup_{n\to\infty} |A_n| \le 1-\sqrt{e}/2 < 0.1757.
\end{align}

Schur remarked that the $\limsup$ in \eqref{1.3} is equal to $0$ for {\em monic} polynomials from $\Z_n(D)$ by Kronecker's theorem \cite{Kr}. We prove that $\lim_{n\to\infty} A_n = 0$ for any sequence of polynomials from Schur's class $\Z_n^s(D,M),\ n\in\N.$ This result is obtained as a consequence of the asymptotic equidistribution of zeros near the unit circle. Namely, if $\{\alpha_{k,n}\}_{k=1}^n$ are the zeros of $P_n$, we define the zero counting measure
\[
\tau_n := \frac{1}{n} \sum_{k=1}^n \delta_{\alpha_{k,n}},
\]
where $\delta_{\alpha_{k,n}}$ is the unit point mass at $\alpha_{k,n}$. Consider the normalized arclength measure $\mu_D$ on the unit circumference, with $d\mu_D(e^{it}):=\frac{1}{2\pi}dt.$ If $\tau_n$ converge weakly to $\mu_D$ as $n\to\infty$ ($\tau_n \stackrel{*}{\rightarrow} \mu_D$) then
\[
\lim_{n\to\infty} A_n = \lim_{n\to\infty} \int z\,d\tau_n(z) = \int z\,d\mu_D(z) = 0.
\]
Thus Schur's problem is solved by the following result \cite{Pr3}.

\begin{theorem} \label{thm1.1}
If $P_n\in\Z_n^s(D,M),\ n\in\N,$ then $\tau_n \stackrel{*}{\rightarrow} \mu_D$ as $\deg(P_n)=n\to\infty.$
\end{theorem}
In fact, Theorem \ref{thm1.1} is a simple consequence of more general results from Section 2. Ideas on the equidistribution of zeros date back to the work of Jentzsch \cite{Je} on the asymptotic zero distribution of the partial sums of a power series, and its generalization by Szeg\H{o} \cite{Sz}. They were developed further by Erd\H{o}s and Tur\'an \cite{ET}, and many others, see Andrievskii and Blatt \cite{AB} for history and additional references. More recently, this topic received renewed attention in number theory, e.g. in the work of Bilu \cite{Bi}, Bombieri \cite{Bo} and Rumely \cite{Ru}.

Theorems A and B were developed in the following directions. If $P_n(z)=a_n\prod_{k=1}^n (z-\alpha_{k,n})$ is irreducible over integers, then $\{\alpha_{k,n}\}_{k=1}^n$ is called a complete set of conjugate algebraic numbers of degree $n$. When $a_n=1$, we refer to $\{\alpha_{k,n}\}_{k=1}^n$ as algebraic integers. If $\alpha=\alpha_{1,n}$ is one of the conjugates, then the sum of $\{\alpha_{k,n}\}_{k=1}^n$ is also called the trace $\tr(\alpha)$ of $\alpha$ over $\Q$. Siegel \cite{Si} improved Theorem B for totally positive algebraic integers to
\[
\liminf_{n\to\infty} A_n = \liminf_{n\to\infty} \tr(\alpha)/n > 1.7336105,
\]
by using an ingenious refinement of the arithmetic-geometric means
inequality that involves the discriminant of $\alpha_{k,n}$. Smyth
\cite{Sm2} introduced a numerical method of ``auxiliary
polynomials," which produced a series of subsequent improvements of
the above lower bound. The original papers \cite{Sm1,Sm2} contain the bound $1.7719$. The most recent results include bounds $1.780022$ by Aguirre, Bilbao, and Peral \cite{ABP}, $1.784109$ by Aguirre and Peral \cite{AP}, and $1.78702$ by Flammang. Thus the Schur-Siegel-Smyth trace problem is to find the smallest limit point $\ell$ for the set of values of mean traces $A_n$ for all totally positive and real algebraic integers. It was observed by Schur \cite{Sch} (see also Siegel \cite{Si}), that $\ell \le 2$. This immediately follows from the fact that, for any odd prime $p$, the totally positive algebraic integer $4\cos^2(\pi/p)$ has degree $(p-1)/2$ and trace $p-2.$ The Schur-Siegel-Smyth trace problem is probably the best known unsolved problem that originated in \cite{Sch}. It is connected with other extremal problems for polynomials with integer coefficients, such as the integer Chebyshev problem, see Borwein and Erd\'elyi \cite{BE}, Borwein \cite{Bor1}, Flammang, Rhin, and Smyth \cite{FRS}, Pritsker \cite{Pr1}, Aguirre and Peral \cite{AP}, and Smyth \cite{Sm3}. Other developments on Schur's problems for the means of algebraic numbers may be found in the papers by Dinghas \cite{Di} and Hunter \cite{Hu}. Although we are not able to provide a complete solution to the Schur-Siegel-Smyth trace problem by finding the smallest values of $\liminf$ in Theorems A and B, we give the sharp lower bound (namely 2) in certain important special cases. Our results are based again on the limiting distribution of algebraic numbers in subsets of the real line, see Section 2.

Section 3 is devoted to the quantitative aspects of convergence $\tau_n \stackrel{*}{\rightarrow} \mu_D$ as $n\to\infty$. We prove a new version (and generalization) of the Erd\H{o}s-Tur\'an theorem on equidistribution of zeros near the unit circle, and near more general sets. This gives estimates of convergence rates for $A_n$ and $S_n$ in Schur's problems.
Furthermore, we obtain some unexpected estimates on growth of polynomials with integer coefficients as an application.

All proofs are given in Section 4.

\section{Asymptotic distribution of algebraic numbers}

We consider asymptotic zero distribution for polynomials with integer coefficients that have sufficiently small norms on compact sets. Asymptotic zero distribution of polynomials is a classical area of analysis with long history that started with papers of Jentzsch \cite{Je} and Szeg\H{o} \cite{Sz}, see \cite{AB} for more complete bibliography. Most of the results developed in analysis use the supremum norms of polynomials. However, the use of the supremum norm even for Schur's problem on the unit disk represents an immediate difficulty, as we have no suitable estimates for polynomials from the class $\Z_n^s(D,M)$. A better way to measure the size of an integer polynomial on the unit disk is given by the Mahler measure, which is also known as the $L_0$ norm or the geometric mean. The Mahler measure of a polynomial $P_n(z) = a_n\prod_{k=1}^n (z-\alpha_{k,n}),\ a_n\neq 0,$ is defined by
\[
M(P_n) := \exp\left(\frac{1}{2\pi} \int_0^{2\pi} \log |P_n(e^{it})|\,dt\right).
\]
Note that $M(P_n) = \lim_{p\to 0} \|P_n\|_p$, where $\|P_n\|_p:=\left(\frac{1}{2\pi} \int_0^{2\pi} |P_n(e^{it})|^p\,dt\right)^{1/p},\ p>0$, hence the $L_0$ norm name. We caution, however, that the Mahler measure does not satisfy the triangle inequality. Jensen's formula readily gives \cite[p. 3]{Bo}
\[
M(P_n) = |a_n|\prod_{k=1}^n \max(1,|\alpha_{k,n}|).
\]
It is immediate to see now that $M(P_n)=|a_n|\le M$ for any $P_n\in\Z_n(D,M),$ which illustrates usefulness and convenience of the Mahler measure for our purposes. Ideas connecting the Mahler measure and distribution of algebraic numbers are very basic to the area, and they previously appeared in various forms in many papers. Without trying to present a comprehensive survey, we mention results on the lower bounds
for the Mahler measure by Schinzel \cite{Sc}, Langevin
\cite{La1, La2, La3}, Mignotte \cite{Mi}, Rhin and Smyth
\cite{RhSm}, Dubickas and Smyth \cite{DuSm}, and the recent survey of Smyth \cite{Sm4}. The asymptotic distribution of algebraic numbers near the unit circle was considered by Bilu \cite{Bi} (see also Bombieri \cite{Bo}) in terms of the absolute logarithmic (or na\"ive) height. His results were generalized to compact sets of capacity 1 by Rumely \cite{Ru}. We proceed to a similar generalization, but use a somewhat different notion of the generalized Mahler measure to obtain an ``if and only if" theorem on the equidistribution of algebraic numbers near arbitrary compact sets in the plane. Our proofs follow standard potential theoretic arguments, and are relatively simple and short.

Consider an arbitrary compact set $E\subset\C.$ As a normalization for its size, we assume that capacity cap$(E)=1,$ see \cite{Ts}, p. 55. In particular, cap$(D)=1$ and capacity of a segment is equal to one quarter of its length \cite{Ts}, p. 84. Examples of sets of capacity one on the real line are given by the segments $[-2,2]$ and $[0,4]$. Let $\mu_E$ be the equilibrium measure of $E$ \cite{Ts}, p. 55, which is a unique probability measure expressing the steady state distribution of charge on the conductor $E$. Note that $\mu_E$ is supported on the boundary of the unbounded connected component $\Omega_E$ of $\overline{\C}\setminus E$ by \cite{Ts}, p. 79. For the unit disk $D$, the equilibrium measure $d\mu_D(e^{it})=\frac{1}{2\pi}dt$ is the normalized arclength measure on the unit circumference. We also have (cf. \cite{ST}, p. 45) that
\begin{align*}
d\mu_{[-2,2]}(x)=\frac{dx}{\pi\sqrt{4-x^2}}, \ x\in(-2,2), \quad \mbox{and}\quad d\mu_{[0,4]}(x)=\frac{dx}{\pi\sqrt{x(4-x)}}, \ x\in(0,4).
\end{align*}

Consider the Green function $g_E(z,\infty)$ for $\Omega_E$ with pole at $\infty$ (cf. \cite{Ts}, p. 14), which is a positive harmonic function in $\Omega_E\setminus\{\infty\}$. Note that $g_D(z,\infty)=\log|z|,\ |z|>1,$ and $g_{[-2,2]}(z,\infty)=\log|z+\sqrt{z^2-4}|-\log{2},\ z\in\C\setminus[-2,2].$ Thus a natural generalization of the Mahler measure for $P_n(z) = a_n\prod_{k=1}^n (z-\alpha_{k,n}),\ a_n\neq 0,$ on an arbitrary compact set $E$ of capacity 1, is given by
\[
M_E(P_n) := |a_n| \exp\left(\sum_{\alpha_{k,n}\in\Omega_E} g_E(\alpha_{k,n},\infty)\right).
\]
If no $\alpha_{k,n}\in\Omega_E$ then we assume that the above (empty) sum is equal to zero. In the sequel, any empty sum is equal to 0, and any empty product is equal to 1 by definition.

We are now ready to state the main equidistribution result.
\begin{theorem} \label{thm2.1}
Let $P_n(z) = a_n\prod_{k=1}^n (z-\alpha_{k,n}),\ \deg(P_n)=n\in\N,$ be a sequence of polynomials with integer coefficients and simple zeros. Suppose that $E\subset\C$ is a compact set of capacity $1$. We have
\begin{align} \label{2.1}
\lim_{n\to\infty} \left(M_E(P_n)\right)^{1/n} = 1
\end{align}
if and only if
\begin{align} \label{2.2}
\left\{
\begin{array}{l}
(i) \dis\lim_{n\to\infty} |a_n|^{1/n} = 1, \\ (ii) \dis\lim_{R\to\infty} \lim_{n\to\infty} \left( \prod_{|\alpha_{k,n}|\ge R} |\alpha_{k,n}| \right)^{1/n} = 1, \\ (iii)\ \dis\tau_n = \frac{1}{n}\sum_{k=1}^n \delta_{\alpha_{k,n}} \stackrel{*}{\rightarrow} \mu_E \mbox{ as }n\to\infty.
\end{array}
\right.
\end{align}
\end{theorem}
\noindent{\bf Remark.} Our proof shows that for $E=D$ one can replace ({\it ii}) in \eqref{2.2} by the condition: There exists $R>1$ such that
\[
\lim_{n\to\infty} \left( \prod_{|\alpha_{k,n}|\ge R} |\alpha_{k,n}| \right)^{1/n} = 1.
\]
In the direction \eqref{2.1} $\Rightarrow$ \eqref{2.2}({\it iii}), our result essentially reduces to that of Bilu \cite{Bi} for the unit disk, and to the result of Rumely \cite{Ru} for general compact sets. Indeed, if $P_n$ is the minimal (irreducible) polynomial for the complete set of conjugate algebraic numbers $\{\alpha_{k,n}\}_{k=1}^n$, then the logarithmic height $h(\alpha_n) = \frac{1}{n} \log M(P_n)$ by \cite{Lan}, p. 54. Hence \eqref{2.1} gives that $h(\alpha_n)\to 0$, which is a condition used by Bilu \cite{Bi}. The converse direction \eqref{2.2} $\Rightarrow$ \eqref{2.1} seems to be new even in the unit disk case. Clearly, Theorem \ref{thm1.1} is a simple consequence of Theorem \ref{thm2.1}.

When the leading coefficients of polynomials are bounded, and all zeros are located in $E$, as assumed by Schur, then we can allow certain multiple zeros. Define the multiplicity of an irreducible factor $Q$ (with integer coefficients) of $P_n$ as an integer $m_n\ge 0$ such that $Q^{m_n}$ divides $P_n$, but $Q^{m_n+1}$ does not divide $P_n$. If a factor $Q$ occurs infinitely often in a sequence $P_n,\ n\in\N,$ then $m_n=o(n)$ means $\lim_{n\to\infty} m_n/n =0.$ If $Q$ is present only in finitely many $P_n$, then $m_n=o(n)$ by definition. We note that any infinite sequence of distinct factors $Q_k$ of polynomials $P_n\in\Z_n(E,M)$ must satisfy $\deg(Q_k)\to\infty$ as $k\to\infty.$ Indeed, if the degrees of $Q_k$ are uniformly bounded, then Vi\`ete's formulas expressing coefficients through the symmetric functions of zeros lead to a uniform bound on all coefficients, where we also use the uniform bounds on the leading coefficients and all zeros for $P_n\in\Z_n(E,M)$. This means that we may only have finitely many such factors $Q_k$ of bounded degree.

\begin{theorem} \label{thm2.2}
Let $E\subset\C$ be a compact set of capacity $1$. Assume that $P_n\in\Z_n(E,M),\ n\in\N$. If every irreducible factor in the sequence of polynomials $P_n$ has multiplicity $o(n)$, then $\tau_n \stackrel{*}{\rightarrow} \mu_E$ as $n\to\infty.$
\end{theorem}

We state a simple corollary that includes a solution of Schur's problem for the unit disk, cf. Theorem C. This result was announced in \cite{Pr3}, together with special cases of other results from this section stated for the unit disk.
\begin{corollary} \label{cor2.3}
Suppose that $E=D.$ If $P_n(z)=a_n\prod_{k=1}^n (z-\alpha_{k,n}),\ \deg(P_n)=n\in\N,$ satisfy $\tau_n \stackrel{*}{\rightarrow} \mu_D$ when $n\to\infty$, in the settings of Theorem \ref{thm2.1} or \ref{thm2.2}, then
\[
\lim_{n\to\infty} \frac{1}{n}\sum_{k=1}^n \alpha_{k,n}^m =0, \quad m\in\N.
\]
\end{corollary}
We also show that the uniform norms
\[
\|P_n\|_E:= \sup_{z\in E} |P_n(z)|
\]
have at most subexponential growth on regular sets $E$, under the assumptions of Theorem \ref{thm2.1}. Regularity is understood here in the sense of the Dirichlet problem for $\Omega_E$, which means that the limiting boundary values of $g_E(z,\infty)$ are all zero, see \cite{Ts}, p. 82.
\begin{theorem} \label{thm2.4}
Let $E\subset\C$ be a regular compact set of capacity $1$. If $P_n,\ \deg(P_n)=n\in\N,$ is a sequence of polynomials with integer coefficients and simple zeros, then
\begin{align} \label{2.3}
\lim_{n\to\infty} \|P_n\|_E^{1/n} = 1
\end{align}
is equivalent to \eqref{2.1} or \eqref{2.2}.
\end{theorem}
This result is somewhat unexpected, as we have no direct control of the supremum norm or coefficients (except for the leading one). For example, $P_n(z)=(z-1)^n$ has the norm $\|P_n\|_D=2^n$, but $M(P_n)=1.$ Theorem \ref{thm2.4} also indicates close connections with the results on the asymptotic zero distribution developed in analysis, see \cite{AB} for many references, where use of the supremum norm is standard. Another easy example $E=D\cup\{z=2\}$ and $P_n(z)=z^n-1,\ n\ge 2,$ shows that the regularity assumption cannot be dropped. Indeed, we have $\|P_n\|_E = 2^n-1$ but $M_E(P_n) = 1$ in this case (observe the single irregular point $z=2$).

We now turn to algebraic numbers on the real line, see Theorems A and B. Combining Theorem \ref{thm2.1} with the results of Baernstein, Laugesen and Pritsker \cite{BLP}, we obtain sharp lower bounds in the following special cases of Schur's problems on the means of totally real and totally positive algebraic numbers.

\begin{corollary} \label{cor2.5}
Let $P_n(z) = a_n\prod_{k=1}^n (z-\alpha_{k,n})\in\Z_n^s(\R),\ \deg(P_n)=n\in\N,$ be a sequence of polynomials, and let $\phi:\R\to\R_+$ be convex. Suppose that $E\subset\R$ is a compact set of capacity $1$ symmetric about the origin. If $\lim_{n\to\infty} \left(M_E(P_n)\right)^{1/n} = 1$ then
\[
\liminf_{n\to\infty} \frac{1}{n} \sum_{k=1}^n \phi(\alpha_{k,n}) \ge \int_{-2}^2
\frac{\phi(x)\,dx}{\pi\sqrt{4-x^2}}.
\]
In particular,
\[
\liminf_{n\to\infty} \frac{1}{n} \sum_{k=1}^n \alpha_{k,n}^2 \ge \int_{-2}^2 \frac{x^2\,dx}{\pi\sqrt{4-x^2}} = 2.
\]
\end{corollary}
The latter inequality should be compared with Theorem A. Note that the bound 2 is asymptotically attained by the zeros of the Chebyshev
polynomials $t_n(x):=2\cos(n\arccos(x/2))$ for the segment $[-2,2]$,
which are the monic polynomials of smallest supremum norm on $[-2,2].$ It is known that these polynomials have integer coefficients, and that
$t_n(x)/x$ are irreducible for any odd prime $n=p$, cf. \cite{Sch}
and \cite{Ri}, p. 228.

We next state the corresponding result for the totally positive case (Schur-Siegel-Smyth trace problem).

\begin{corollary} \label{cor2.6}
Let $P_n(z) = a_n\prod_{k=1}^n (z-\alpha_{k,n})\in\Z_n^s(\R_+),\ \deg(P_n)=n\in\N,$ be a sequence of polynomials. Suppose that $E\subset\R_+$ is a compact set of capacity $1$. We also assume that $\phi:\R_+\to\R_+$, and that $\phi(x^2)$ is convex on $\R$. If $\lim_{n\to\infty} \left(M_E(P_n)\right)^{1/n} = 1$ then
\[
\liminf_{n\to\infty} \frac{1}{n} \sum_{k=1}^n \phi(\alpha_{k,n}) \ge \int_0^4 \frac{\phi(x)\,dx}{\pi\sqrt{x(4-x)}}.
\]
Setting $\phi(x)=x^m,\ m\in\N,$ we obtain
\[
\liminf_{n\to\infty} \frac{1}{n} \sum_{k=1}^n \alpha_{k,n}^m \ge \int_0^4 \frac{x^m\,dx}{\pi\sqrt{x(4-x)}} = 2^m \frac{1\cdot 3\cdot\ldots\cdot(2m-1)}{m!}.
\]
\end{corollary}
Thus the limit of the arithmetic means $A_n$ under the assumptions of Corollary \ref{cor2.6} is equal to the optimal bound 2, cf. Theorem B.
A possible application for both Corollaries \ref{cor2.5} and \ref{cor2.6} is the case when $E$ satisfies the corresponding assumptions, and $P_n\in\Z_n^s(E,M),\ n\in\N,$ so that \eqref{2.1} is easily verified. Note, however, that $\cup_{n=1}^{\infty} \Z_n^s(E,M)$ may be finite (or even empty) for some sets of capacity one. It is a nontrivial question for which sets $E$ the set of polynomials $\cup_{n=1}^{\infty} \Z_n^s(E,M)$ is infinite, see e.g. the work of Robinson \cite{Ro1}--\cite{Ro3}, and of Dubickas and Smyth \cite{DuSm}--\cite{DuSm2}.

\section{Rate of convergence and discrepancy in equidistribution}

We now consider the quantitative aspects of the convergence $\tau_n \stackrel{*}{\rightarrow} \mu_E$, starting with the case $E=D$. As an application, we obtain estimates of the convergence rate for $A_n$ to $0$ in Schur's problem for the unit disk. A classical quantitative result on the distribution of zeros near the unit circle is due to Erd\H{o}s and Tur\'an \cite{ET}. For $P_n(z) = \sum_{k=0}^n a_k z^k$ with $a_k\in\C,$ let $N(\phi_1,\phi_2)$ be the number of zeros in the sector $\{z\in\C:0\le \phi_1 \le \arg(z) \le \phi_2< 2\pi\},$ where $\phi_1 < \phi_2.$ Erd\H{o}s and Tur\'an \cite{ET} proved that
\begin{align} \label{3.1}
\left|\frac{N(\phi_1,\phi_2) }{n} - \frac{\phi_2-\phi_1}{2\pi}\right| \le 16 \sqrt{\frac{1}{n}\log\frac{\|P_n\|_D}{\sqrt{|a_0 a_n|}}}.
\end{align}
The constant $16$ was improved by Ganelius \cite{Ga}, and $\|P_n\|_D$ was replaced by weaker integral norms by Amoroso and Mignotte \cite{AM}, see \cite{AB} for more history and references. Our main difficulty in applying \eqref{3.1} to Schur's problem is absence of an effective estimate for $\|P_n\|_D,\ P_n\in\Z_n^s(D,M)$. We prove a new ``discrepancy" estimate via energy considerations from potential theory. These ideas originated in part in the work of Kleiner \cite{Kl}, and were developed by Sj\"{o}gren \cite{Sj} and Huesing \cite{Hus}, see \cite{AB}, Ch. 5.

\begin{theorem} \label{thm3.1}
Let $\phi:\C\to\R$ satisfy $|\phi(z)-\phi(t)|\le A|z-t|,\ z,t\in\C,$ and $\supp(\phi)\subset\{z:|z|\le R\}.$ If $P_n(z) = a_n\prod_{k=1}^n (z-\alpha_{k,n}),\ a_n\neq 0,$ is a polynomial with integer coefficients and simple zeros, then
\begin{align} \label{3.2}
\left|\frac{1}{n}\sum_{k=1}^n \phi(\alpha_{k,n}) - \int\phi\,d\mu_D\right| \le A(2R+1) \sqrt{\frac{\log\max(n,M(P_n))}{n}}, \quad n\ge 55.
\end{align}
\end{theorem}
This theorem is related to the recent results of Favre and Rivera-Letelier \cite{FR}, obtained in terms of adelic heights on the projective line (see Theorem 5 in the original paper, and note corrections in the Corrigendum). An earlier result of Petsche \cite{Pe}, stated in terms of the Weil height, contains a weaker estimate than \eqref{3.2}. Our approach gives a result for arbitrary polynomials with simple zeros, and for any continuous $\phi$ with the finite Dirichlet integral $D[\phi]=\iint(\phi_x^2 +\phi_y^2)\,dA$, cf. Theorem \ref{thm5.2}. Moreover, it is extended in Theorem \ref{thm5.3} to more general sets of logarithmic capacity $1$, e.g. to $[-2,2]$.  These results have a number of applications to the problems on integer polynomials considered in \cite{Bor1}.

Choosing $\phi$ appropriately, we obtain an estimate of the means $A_n$ in Schur's problem for the unit disk.
\begin{corollary} \label{cor3.2}
If $P_n(z) = a_n\prod_{k=1}^n (z-\alpha_{k,n})\in\Z_n^s(D,M)$ then
\[
\left|\frac{1}{n}\sum_{k=1}^n \alpha_{k,n}\right| \le 8 \sqrt{\frac{\log{n}}{n}},\quad n\ge \max(M,55).
\]
\end{corollary}

Observe that \eqref{2.3} is granted for Schur's class $\Z_n^s(D,M)$ by Theorem \ref{thm2.4}. We now state an improvement in the following estimate.
\begin{corollary} \label{cor3.3}
If $P_n\in\Z_n^s(D,M)$ then there exists an absolute constant $c>0$ such that
\[
\|P_n\|_D \le e^{c\sqrt{n}\log{n}},\quad n\ge \max(M,2).
\]
\end{corollary}

We are passing to sets on the real line and totally real algebraic numbers. It is certainly possible to consider quite general sets in the plane from the viewpoint of potential theoretic methods, see Theorem \ref{thm5.3}. However, we restrict ourselves to the sets that are most interesting in number theory. This also helps to avoid certain unnecessary technical difficulties.

\begin{theorem} \label{thm3.4}
Let $E=[a,b]\subset\R,\ b-a=4.$ Suppose that $\phi:\C\to\R$ satisfy $|\phi(z)-\phi(t)|\le A|z-t|,\ z,t\in\C,$ and $\supp(\phi)\subset\{z:|z-(a+2)|\le R\}.$ If $P_n(z) = a_n\prod_{k=1}^n (z-\alpha_{k,n}),\ a_n\neq 0,$ is a polynomial with integer coefficients and simple zeros, then
\begin{align} \label{3.3}
\left|\frac{1}{n}\sum_{k=1}^n \phi(\alpha_{k,n}) - \int\phi\,d\mu_{[a,b]}\right| \le A(3R+1) \sqrt{\frac{\log\max(n,M_{[a,b]}(P_n))}{n}}, \quad n\ge 25.
\end{align}
\end{theorem}
One should compare this result with a classical discrepancy theorem of Erd\H{o}s and Tur\'an \cite{ET2} for the segment $[-1,1],$ and more recent work surveyed in \cite{AB}. Recall that $g_{[a,b]}(z,\infty)=\log|z-(a+b)/2+\sqrt{(z-a)(z-b)}|-\log{2},\ z\in\C\setminus[a,b],\ b-a=4,$ and
\[
d\mu_{[a,b]}(x)=\frac{dx}{\pi\sqrt{(x-a)(b-x)}},\quad x\in(a,b).
\]
We state consequences of Theorem \ref{thm3.4} for the means of algebraic numbers, and for the growth of the supremum norms of polynomials with integer coefficients on segments.

\begin{corollary} \label{cor3.5}
Let $E=[a,b]\subset\R,\ b-a=4.$ If $P_n(z) = a_n\prod_{k=1}^n (z-\alpha_{k,n})\in\Z_n^s([a,b],M)$ then
\[
\left|\frac{1}{n}\sum_{k=1}^n \alpha_{k,n} - \frac{a+b}{2}\right| \le 6\,\max(|a|,|b|)\, \sqrt{\frac{\log{n}}{n}},\quad n\ge \max(M,25).
\]
\end{corollary}

\begin{corollary} \label{cor3.6}
If $P_n(z) = a_n\prod_{k=1}^n (z-\alpha_{k,n})\in\Z_n^s([-2,2],M)$ then
\[
\left|\frac{1}{n}\sum_{k=1}^n \alpha_{k,n}^2 - 2\right| \le 24 \sqrt{\frac{\log{n}}{n}},\quad n\ge \max(M,25).
\]
\end{corollary}

\begin{corollary} \label{cor3.7}
If $P_n\in\Z_n^s([-2,2],M)$ then there exists an absolute constant $c>0$ such that
\[
\|P_n\|_{[-2,2]} \le e^{c\sqrt{n}\log{n}},\quad n\ge \max(M,2).
\]
\end{corollary}

It is an interesting question whether the rates in terms of $n$ can be improved in the results of this section. Erd\H{o}s and Tur\'an \cite{ET} constructed an example that shows \eqref{3.1} gives a correct rate in their setting, but that example is based on a sequence of polynomials with multiple zeros. After the original version of this paper was written, the author was able to show that Corollaries \ref{cor3.2} and \ref{3.3} are sharp up to the logarithmic factors. Constructed examples are based on products of cyclotomic polynomials, see Example 2.8 in \cite{Pr4}. However, it is plausible that our rates can be substantially improved for the sequences of irreducible polynomials.

\section{Proofs}

We start with a brief review of basic facts from potential theory. A complete account may be found in the books by Ransford \cite{Ra}, Tsuji \cite{Ts}, and Landkof \cite{La}. For a Borel measure $\mu$ with compact support, define its potential by
\[
U^{\mu}(z):=\int \log\frac{1}{|z-t|}\,d\mu(t), \quad z\in\C,
\]
see \cite{Ts}, p. 53. It is known that $U^{\mu}(z)$ is a superharmonic function in $\C$, which is harmonic outside $\supp(\mu)$. We shall often use the identity
\[
\log|P_n(z)| = \log|a_n| - n U^{\tau_n}(z),
\]
where $P_n(z) = a_n\prod_{k=1}^n (z-\alpha_{k,n})$ and $\tau_n = \frac{1}{n} \sum_{k=1}^n \delta_{\alpha_{k,n}}.$ The energy of a Borel measure $\mu$ is defined by
\[
I[\mu]:= \iint \log \dis\frac{1}{|z-t|} \, d \mu(t)d \mu(z) = \int U^{\mu}(z)\,d\mu(z),
\]
cf. \cite{Ts}, p. 54. For a compact set $E\subset\C$ of positive capacity, the minimum energy among all probability measures supported on $E$ is attained by the equilibrium measure $\mu_E$, see \cite{Ts}, p. 55. If $U^{\mu_E}(z)$ is the equilibrium (conductor) potential for a compact set $E$ of capacity $1$, then Frostman's theorem (cf. \cite{Ts}, p. 60) gives that
\begin{align} \label{5.1}
U^{\mu_E}(z) \le 0,\ z\in\C, \quad \mbox{and} \quad U^{\mu_E}(z) = 0 \mbox{ q.e. on }E.
\end{align}
The second statement means that equality holds quasi everywhere on $E$, i.e. except for a subset of zero capacity in $E$. This may be made even more precise, as $U^{\mu_E}(z) = 0$ for any $z\in\overline\C\setminus \overline\Omega_E,$ where $\Omega_E$ is the unbounded connected component of $\overline\C\setminus E$. Hence $U^{\mu_E}(z) = 0$ for any $z$ in the interior of $E$ by \cite{Ts}, p. 61. Furthermore, $U^{\mu_E}(z) = 0$ for $z\in\partial\Omega_E$ if and only if $z$ is a regular point for the Dirichlet problem in $\Omega_E$, see \cite{Ts}, p. 82. We mention a well known connection of the equilibrium potential for $E,\ {\rm cap}(E)=1$, with the Green function $g_E(z,\infty)$ for $\Omega_E$ with pole at $\infty$:
\begin{align} \label{5.2}
g_E(z,\infty) = - U^{\mu_E}(z),\quad z\in\C.
\end{align}
This gives a standard extension of $g_E(z,\infty)$ from $\Omega_E$ to the whole plane $\C,$ see \cite{Ts}, p. 82. Thus $g_E(z,\infty)=0$ for quasi every $z\in\partial\Omega_E$, and $g_E(z,\infty)=0$ for any $z\in\overline\C\setminus\overline\Omega_E,$ by \eqref{5.1} and \eqref{5.2}. For a polynomial $P_n(z) = a_n\prod_{k=1}^n (z-\alpha_{k,n})$, we may define a slightly different generalization of the Mahler measure by
\begin{align} \label{5.3}
\tilde M_E(P_n) = \exp\left(\int\log|P_n(z)|\,d\mu_E(z)\right).
\end{align}
One observes from \eqref{5.2} that
\begin{align} \label{5.4}
\log \tilde M_E(P_n) = \log|a_n| - \sum_{k=1}^n U^{\mu_E}(\alpha_{k,n}) = \log|a_n| + \sum_{k=1}^n g_E(\alpha_{k,n},\infty).
\end{align}
Since $g_E(z,\infty)\ge 0,\ z\in\C,$ it follows immediately that
\begin{align} \label{5.5}
M_E(P_n) \le \tilde M_E(P_n).
\end{align}
Furthermore, we have equality in \eqref{5.5} for regular sets $E$ because $g_E(z,\infty)=0,\ z\in\overline\C\setminus \Omega_E.$

\subsection{Proofs for Sections 1 and 2}

\begin{proof}[Proof of Theorem \ref{thm1.1}]
This result follows immediately from Theorem \ref{thm2.1}, as
$M(P_n) = |a_n| \le M$ for $P_n\in\Z_n^s(D,M)$, and \eqref{2.1} is satisfied.
\end{proof}

\begin{proof}[Proof of Theorem \ref{thm2.1}]
We first prove that \eqref{2.1} implies \eqref{2.2}. Since $|a_n|\ge 1$ and $g_E(\alpha_{k,n},\infty)>0,\ \alpha_{k,n}\in\Omega_E,$ equation \eqref{2.2}({\it i}) is a consequence of \eqref{2.1} and the definition of $M_E(P_n)$. Suppose that $R>0$ is sufficiently large, so that $E\subset D_R:=\{z:|z|<R\}.$
Then we have that
\[
0 \le \frac{1}{n} \sum_{|\alpha_{k,n}|\ge R} g_E(\alpha_{k,n},\infty) \le \frac{1}{n} \log M_E(P_n) \to 0 \quad\mbox{as } n\to\infty.
\]
Hence
\[
\lim_{n\to\infty} \frac{1}{n} \sum_{|\alpha_{k,n}|\ge R} g_E(\alpha_{k,n},\infty) = 0.
\]
Recall that $\lim_{z\to\infty} (g_E(z,\infty)-\log|z|) = -\log{\rm cap}(E) = 0,$ see \cite{Ts}, p. 83. It follows that for any $\eps>0$, there is a sufficiently large $R>0$ such that $-\eps<\log|z|-g_E(z,\infty)<\eps$ for $|z|\ge R$, and
\[
-\eps\le\lim_{n\to\infty} \frac{1}{n} \sum_{|\alpha_{k,n}|\ge R} \log|\alpha_{k,n}|\le\eps.
\]
Therefore, \eqref{2.2}({\it ii}) is proved. In order to show that \eqref{2.2}({\it iii}) holds, we first deduce that each closed set $K\subset\Omega_E$ has $o(n)$ zeros of $P_n$ as $n\to\infty,$ i.e.
\begin{align} \label{5.6}
\lim_{n\to\infty} \tau_n(K) = 0.
\end{align}
This fact follows because $\min_{z\in K} g_E(z,\infty) > 0$ and
\[
0 \le \tau_n(K) \min_{z\in K} g_E(z,\infty) \le  \frac{1}{n} \sum_{\alpha_{k,n}\in K} g_E(\alpha_{k,n},\infty) \le \frac{1}{n} \log M_E(P_n) \to 0 \quad\mbox{as } n\to\infty.
\]
Thus if $R>0$ is sufficiently large, so that $E\subset D_R,$ we have $o(n)$ zeros of $P_n$ in $\C\setminus D_R.$ Consider
\[
\hat\tau_n := \frac{1}{n} \sum_{|\alpha_{k,n}|<R} \delta_{\alpha_{k,n}}.
\]
Since $\supp(\hat\tau_n)\subset D_R,\ n\in\N,$ we use Helly's theorem (cf. \cite{ST}, p. 3) to select a weakly convergent subsequence from the sequence $\hat\tau_n$. Preserving the same notation for this subsequence, we assume that $\hat\tau_n \stackrel{*}{\rightarrow} \tau$ as $n\to\infty$. It is clear from \eqref{5.6} that $\tau_n \stackrel{*}{\rightarrow} \tau$  as $n\to\infty$, and that $\tau$ is a probability measure supported on the compact set $\hat E := \overline{\C}\setminus\Omega_E.$

Let $\Delta(P_n)=a_n^{2n-2} (V(P_n))^2$ be the discriminant of $P_n$, where
\[
V(P_n):=\prod_{1\le j<k\le n} (\alpha_{j,n}-\alpha_{k,n})
\]
is the Vandermonde determinant. Since $P_n$ has integer coefficients, $\Delta(P_n)$ is an integer, see \cite{Pra}, p. 24. As $P_n$ has simple roots, we obtain that $\Delta(P_n)\neq 0$ and $|\Delta(P_n)|\ge 1.$ It follows from \eqref{2.2}({\it i}) that
\begin{align} \label{5.7}
\liminf_{n\to\infty} |V(P_n)|^{\frac{2}{(n-1)n}} \ge 1.
\end{align}
Suppose that $R>0$ is large, and order $\alpha_{k,n}$ as follows
\[
|\alpha_{1,n}| \le |\alpha_{2,n}| \le \ldots \le |\alpha_{m_n,n}| < R \le |\alpha_{m_n+1,n}| \le \ldots \le |\alpha_{n,n}|.
\]
Let $\hat P_n(z) := a_n \prod_{k=1}^{m_n} (z-\alpha_{k,n}),$ so that
$V(\hat P_n)=\prod_{1\le j<k\le m_n} (\alpha_{j,n}-\alpha_{k,n}).$
Hence
\begin{align} \label{5.8}
|V(P_n)|^2  &= |V(\hat P_n)|^2 \prod_{1\le j<k \atop m_n<k\le n} |\alpha_{j,n}-\alpha_{k,n}|^2 \le |V(\hat P_n)|^2 \prod_{m_n<k\le n} (2|\alpha_{k,n}|)^{2(n-1)} \\ \nonumber &\le |V(\hat P_n)|^2 4^{(n-1)(n-m_n)} \left(\prod_{m_n<k\le n} |\alpha_{k,n}|\right)^{2(n-1)},
\end{align}
where we used that $|\alpha_{j,n}-\alpha_{k,n}| \le 2\max(|\alpha_{j,n}|,|\alpha_{k,n}|).$ Note that $\lim_{n\to\infty} m_n/n = 1$. For any $\eps>0$, we find $R>0$ such that
\[
\limsup_{n\to\infty} \left( \prod_{m_n<k\le n} |\alpha_{k,n}| \right)^{2/n} = \limsup_{n\to\infty} \left( \prod_{|\alpha_{k,n}|\ge R} |\alpha_{k,n}| \right)^{2/n} < 1 + \eps
\]
by \eqref{2.2}({\it ii}). Thus we obtain from \eqref{5.8}, \eqref{5.7} and the above estimate that
\begin{align} \label{5.9}
\liminf_{n\to\infty} |V(\hat P_n)|^{\frac{2}{(n-1)n}} \ge \frac{\liminf_{n\to\infty} |V(P_n)|^{\frac{2}{(n-1)n}}}{\limsup_{n\to\infty} \left( \prod_{m_n<k\le n} |\alpha_{k,n}| \right)^{2/n}} \ge \frac{1}{1+\eps}.
\end{align}

We now follow a standard potential theoretic argument to show that $\tau=\mu_E.$ Let $K_M(z,t) := \min\left(-\log{|z-t|},M\right).$ It is clear that $K_M(z,t)$ is a continuous function in $z$ and $t$
on $\C\times\C$, and that $K_M(z,t)$ increases to
$-\log|z-t|$ as $M\to\infty.$ Using the Monotone
Convergence Theorem and the weak* convergence of $\hat\tau_n\times\hat\tau_n$
to $\tau\times\tau,$ we obtain for the energy of $\tau$ that
\begin{align*}
I[\tau] &=-\iint \log|z-t|\,d\tau(z)\,d\tau(t) =
\lim_{M\to\infty} \left( \lim_{n\to\infty} \iint K_M(z,t)\,
d\hat\tau_n(z)\,d\hat\tau_n(t) \right) \\ &= \lim_{M\to\infty} \left(
\lim_{n\to\infty} \left( \frac{2}{n^2} \sum_{1\le j<k\le m_n} K_M(\alpha_{j,n},\alpha_{k,n}) + \frac{M}{n} \right) \right) \\ &\le
\lim_{M\to\infty} \left( \liminf_{n\to\infty} \frac{2}{n^2}
\sum_{1\le j<k\le m_n} \log\frac{1}{|\alpha_{j,n}-\alpha_{k,n}|} \right) \\ &= \liminf_{n\to\infty} \frac{2}{n^2}
\log\frac{1}{|V(\hat P_n)|} \le \log(1+\eps),
\end{align*}
where \eqref{5.9} was used in the last estimate. Since $\eps>0$ is arbitrary, we conclude that $I[\tau]\le 0$. Recall that $\supp(\tau) \subset \hat E = \overline{\C}\setminus\Omega_E,$ where cap$(\hat E)=$ cap$(E)=1$ and $\mu_{\hat E}=\mu_E$ by \cite{Ts}, pp. 79-80. Note also that $I[\nu]>0$ for any probability measure $\nu\neq\mu_{\hat E},\ \supp(\nu)\subset \hat E$, see \cite{Ts}, pp. 79-80. Hence $\tau=\mu_{\hat E}=\mu_E$ and \eqref{2.2}({\it iii}) follows.

Let us turn to the converse statement \eqref{2.2} $\Rightarrow$ \eqref{2.1}. As in the first part of the proof, we note that
$\lim_{z\to\infty} (g_E(z,\infty)-\log|z|) = 0.$ For any $\eps>0$, we choose $R>0$ so large that $E\subset D_R$ and $|g_E(z,\infty)-\log|z||<\eps$ when $|z|\ge R.$ Thus we have from \eqref{2.2}({\it iii}) that
\[
\frac{1}{n} \sum_{|\alpha_{k,n}|\ge R} g_E(\alpha_{k,n},\infty) \le \frac{1}{n} \sum_{|\alpha_{k,n}|\ge R} \log|\alpha_{k,n}| + \frac{o(n)}{n}\eps.
\]
Increasing $R$ if necessary, we can achieve that
\[
\frac{1}{n} \sum_{|\alpha_{k,n}|\ge R} \log|\alpha_{k,n}| < \eps,
\]
by \eqref{2.2}({\it ii}), which implies that
\begin{align} \label{5.10}
\limsup_{n\to\infty} \frac{1}{n} \sum_{|\alpha_{k,n}|\ge R} g_E(\alpha_{k,n},\infty) \le \eps.
\end{align}
On setting $g_E(z,\infty) = - U^{\mu_E}(z),\ z\in\C,$ we continue $g_E(z,\infty)$ as a subharmonic function in $\C.$ Since $g_E(z,\infty)$ is now upper semi-continuous  in $\C,$ we obtain from \eqref{2.2}({\it iii}) and Theorem 0.1.4 of \cite{ST}, p. 4, that
\begin{align} \label{5.11}
\limsup_{n\to\infty} \frac{1}{n} \sum_{|\alpha_{k,n}| < R} g_E(\alpha_{k,n},\infty) &= \limsup_{n\to\infty} \int_{D_R} g_E(z,\infty)\,d\tau_n(z) \le \int_{D_R} g_E(z,\infty)\,d\mu_E(z) \\ \nonumber &= - \int U^{\mu_E}(z)\,d\mu_E(z) = - I[\mu_E] = 0,
\end{align}
where the last equality follows as the energy $I[\mu_E] = -\log{\rm cap}(E) = 0,$ see \cite{Ts}, p. 55. Observe from the definition of $M_E(P_n)$, \eqref{5.4}-\eqref{5.5} and \eqref{2.2}({\it i}) that
\begin{align*}
0 &\le \limsup_{n\to\infty} \frac{1}{n} \log M_E(P_n) \le \limsup_{n\to\infty} \frac{1}{n} \log \tilde M_E(P_n) \\ \nonumber &\le \limsup_{n\to\infty} \frac{1}{n} \log|a_n| + \limsup_{n\to\infty} \frac{1}{n} \sum_{k=1}^n g_E(\alpha_{k,n},\infty) = \limsup_{n\to\infty} \frac{1}{n} \sum_{k=1}^n g_E(\alpha_{k,n},\infty).
\end{align*}
Combining this estimate with \eqref{5.10} and \eqref{5.11}, we arrive at
\[
0 \le \limsup_{n\to\infty} \frac{1}{n} \log M_E(P_n) \le \limsup_{n\to\infty} \frac{1}{n} \log \tilde M_E(P_n) \le \eps.
\]
We now let $\eps\to 0,$ to obtain that
\begin{align} \label{5.12}
\lim_{n\to\infty} \left(M_E(P_n)\right)^{1/n} = \lim_{n\to\infty} \left(\tilde M_E(P_n)\right)^{1/n} = 1.
\end{align}

\end{proof}

An interesting by-product of the above proof is the following fact.

\begin{proposition} \label{prop5.1}
Let $P_n,\ \deg(P_n)=n\in\N,$ be a sequence of polynomials with integer coefficients and simple zeros. Suppose that $E\subset\C$ is a compact set of capacity $1$. Then \eqref{2.1} holds if and only if
\begin{align} \label{5.13}
\lim_{n\to\infty} \left(\tilde M_E(P_n)\right)^{1/n} = 1.
\end{align}
\end{proposition}

\begin{proof}
The implications \eqref{2.1} $\Rightarrow$ \eqref{2.2} $\Rightarrow$ \eqref{5.12} were established in the proof of Theorem \ref{thm2.1}. Hence \eqref{2.1} implies \eqref{5.13}. The converse is immediate from \eqref{5.5}.
\end{proof}

\begin{proof}[Proof of Theorem \ref{thm2.2}.] Let $\phi\in C(\C)$ have compact support. Note that for any $\epsilon>0$ there are finitely many irreducible factors $Q$ in the sequence $P_n$ such that
\[
\left|\int\phi\,d\tau(Q) - \int\phi\,d\mu_E\right| \ge \epsilon,
\]
where $\tau(Q)$ is the zero counting measure for $Q$. Indeed, if we have an infinite sequence of such $Q_k,\ k\in\N$, then $\deg(Q_k)\to\infty$ as $k\to\infty,$ see the explanation given before Theorem \ref{thm2.2}. However, the fact that $\deg(Q_k)\to\infty$ implies that $\int\phi\,d\tau(Q_k) \to \int\phi\,d\mu_E$ by Theorem \ref{thm2.1}, because $M_E(Q_k)\le M$ gives that $\tau(Q_k)\stackrel{*}{\rightarrow} \mu_E$. Let the total number of such exceptional factors $Q_k$ be $N$. Then we have
\[
\left|n\int\phi\,d\tau_n - n\int\phi\,d\mu_E\right| \le N o(n) \max_{z\in E} \left|\phi(z) - \int\phi\,d\mu_E\right| + (n-N)\epsilon,\ n\in\N.
\]
Hence $\limsup_{n\to\infty} |\int\phi\,d\tau_n - \int\phi\,d\mu_E| \le \epsilon,$ and $\lim_{n\to\infty} \int\phi\,d\tau_n = \int\phi\,d\mu_E$ after letting $\epsilon\to 0.$

\end{proof}

\begin{proof}[Proof of Corollary \ref{cor2.3}.]
Since $\tau_n\stackrel{*}{\rightarrow} \mu_D$, we let $\phi(z)=z^m$ and obtain that
\[
\lim_{n\to\infty} \int z^m\,d\tau_n(z) = \int z^m\,d\mu_D(z) = \frac{1}{2\pi} \int_0^{2\pi} e^{im\theta}\,d\theta = 0.
\]
\end{proof}

\begin{proof}[Proof of Theorem \ref{thm2.4}.]
It is clear from the definitions of $M_E(P_n)$ and $\tilde M_E(P_n)$ that
\begin{align*}
1 \le M_E(P_n) \le \tilde M_E(P_n) \le \|P_n\|_E
\end{align*}
for any polynomial $P_n$ with integer coefficients, and any compact set $E$ of capacity 1, see \eqref{5.1}-\eqref{5.5}. Hence \eqref{2.3} implies \eqref{2.1}.

Conversely, assume that \eqref{2.1} holds true. Then \eqref{2.2} follows by Theorem \ref{thm2.1}. Let $P_n(z) = a_n \prod_{k=1}^{n} (z-\alpha_{k,n}),\ n\in\N.$ For any $\eps>0$, we find $R>0$ such that $E\subset D_R=\{z:|z|<R\}$ and
\[
\lim_{n\to\infty} \left( \prod_{|\alpha_{k,n}|\ge R} |\alpha_{k,n}| \right)^{1/n} < 1 + \eps
\]
by \eqref{2.2}({\it ii}). Since there are $o(n)$ numbers $\alpha_{k,n}$ outside $D_R$ by \eqref{2.2}({\it iii}), and since $\|z-\alpha_{k,n}\|_E \le 2|\alpha_{k,n}|$ for $|\alpha_{k,n}|\ge R$, we obtain that
\begin{align} \label{5.14}
\limsup_{n\to\infty} \left\|\prod_{|\alpha_{k,n}| \ge R} (z-\alpha_{k,n})\right\|_E^{1/n} \le \limsup_{n\to\infty}\, 2^{o(n)/n} \left( \prod_{|\alpha_{k,n}|\ge R} |\alpha_{k,n}| \right)^{1/n} \le 1 + \eps.
\end{align}
Let $\|P_n\|_E=|P_n(z_n)|,\ z_n\in E,$ and assume $\lim_{n\to\infty} z_n = z_0\in E$ by compactness. Define
\[
\hat\tau_n := \frac{1}{n} \sum_{|\alpha_{k,n}|<R} \delta_{\alpha_{k,n}},
\]
and note that $\hat\tau_n \stackrel{*}{\rightarrow} \mu_E$ as $n\to\infty$ by \eqref{2.2}({\it iii}). For the polynomial
\[
\hat P_n(z) := \prod_{|\alpha_{k,n}| < R} (z-\alpha_{k,n}),
\]
we have by the Principle of Descent (Theorem I.6.8 of \cite{ST}) that
\begin{align} \label{5.15}
\limsup_{n\to\infty} |\hat P_n(z_n)|^{1/n} = \limsup_{n\to\infty} \exp \left(-U^{\hat\tau_n}(z_n)\right) \le \exp\left(-U^{\mu_E}(z_0)\right) = 1,
\end{align}
where the last equality is a consequence of Frostman's theorem \eqref{5.1} and the regularity of $E$. It is known that $\|P_n\|_E \ge |a_n| ({\rm cap}(E))^n \ge 1$, see \cite{AB}, p. 16. We use this fact together with \eqref{2.2}({\it i}), \eqref{5.14} and \eqref{5.15} in the following estimate:
\begin{align*}
1 &\le \limsup_{n\to\infty} \|P_n\|_E^{1/n} \le \limsup_{n\to\infty} |a_n|^{1/n}\, \limsup_{n\to\infty} |\hat P_n(z_n)|^{1/n} \limsup_{n\to\infty} \left(\prod_{|\alpha_{k,n}| \ge R} |z_n-\alpha_{k,n}|\right)^{1/n} \\ &\le 1 + \eps.
\end{align*}
Letting $\eps\to 0,$ we obtain \eqref{2.3}.
\end{proof}

\begin{proof}[Proof of Corollary \ref{cor2.5}.]
Theorem \ref{thm2.1} implies that $\tau_n \stackrel{*}{\rightarrow} \mu_E$, so that
\[
\liminf_{n\to\infty} \frac{1}{n} \sum_{k=1}^n \phi(\alpha_{k,n}) \ge \liminf_{n\to\infty} \frac{1}{n} \sum_{|\alpha_{k,n}|<R} \phi(\alpha_{k,n}) = \int \phi(x) \, d\mu_E(x),
\]
where $R>0$ is sufficiently large to satisfy $E\subset D_R.$ The inequality
\[
\int \phi(x) \, d\mu_E(x) \ge \int_{-2}^2
\frac{\phi(x)\,dx}{\pi\sqrt{4-x^2}}
\]
follows from Theorem 1 of \cite{BLP}, as $\int z\, d\mu_E(z)=0$. Letting $\phi(x)=x^2,$ we obtain the second inequality in the statement.
\end{proof}

\begin{proof}[Proof of Corollary \ref{cor2.6}.]
As in the previous proof, Theorem \ref{thm2.1} implies that
\[
\liminf_{n\to\infty} \frac{1}{n} \sum_{k=1}^n \phi(\alpha_{k,n}) \ge
\int \phi(x) \, d\mu_E(x).
\]
We apply the change of variable $x=t^2$, and define the compact set $K=\{t\in\R: t^2\in E\}$. Then $K$ is symmetric about the origin, so that $\int t\, d\mu_K(t)=0$. Furthermore, $d\mu_K(t) = d\mu_E(t^2),\ t\in K$, and cap$(K)=1$; see \cite{Ra}, p. 134. The inequalities of Corollary \ref{cor2.6} are now immediate from Theorem 1 of \cite{BLP}, because
\[
\int_E \phi(x) \, d\mu_E(x) = \int_K \phi(t^2) \, d\mu_K(t).
\]
\end{proof}

\subsection{Proofs for Section 3}

It is clear that our estimate \eqref{3.2} measures the difference (discrepancy) between $\tau_n$ and $\mu_D$ is terms of the weak* convergence. Thus we consider a class of continuous test functions $\phi:\R^2\to\R$ with compact supports in the plane $\R^2=\C.$ Let
\[
\omega_{\phi}(r):=\sup_{|z-\zeta|\le r} |\phi(z)-\phi(\zeta)|
\]
be the modulus of continuity of $\phi$ in $\C$. We also require that the functions $\phi$ have finite Dirichlet integrals
\[
D[\phi]:= \iint_{\R^2} \left(\phi_x^2+\phi_y^2\right)\,dxdy,
\]
where it is assumed that the partial derivatives $\phi_x$ and $\phi_y$ exist a.e. on $\R^2$ in the sense of the area measure.

\begin{theorem} \label{thm5.2}
Let $P_n(z)=a_n \prod_{k=1}^{n} (z-\alpha_{k,n}),\ a_n\neq 0,$ be a polynomial with simple zeros. Suppose that $\phi:\C\to\R$ is a continuous function with compact support in the plane, and $D[\phi]<\infty.$ Then for any $r>0$, we have
\begin{align} \label{5.16}
&\left|\int\phi\,d\tau_n - \int\phi\,d\mu_D\right| \\ \nonumber &\le \omega_{\phi}(r) + \sqrt{\frac{D[\phi]}{2\pi}}\,\left(\frac{2}{n}\log M(P_n) - \frac{1}{n^2}\log\left|a_n^2 \Delta(P_n)\right| - \frac{1}{n}\log{r} + 4r\right)^{1/2}.
\end{align}
\end{theorem}

\begin{proof}
Given $r>0$, define the measures $\nu_k^r$ with $d\nu_k^r(\alpha_{k,n} + re^{it}) = dt/(2\pi),\ t\in[0,2\pi).$ Let
\[
\tau_n^r:=\frac{1}{n}\sum_{k=1}^n \nu_k^r,
\]
and estimate
\begin{align} \label{5.17}
\left|\int\phi\,d\tau_n - \int\phi\,d\tau_n^r\right| \le \frac{1}{n}\sum_{k=1}^n \frac{1}{2\pi}\int_0^{2\pi} \left|\phi(\alpha_{k,n}) - \phi(\alpha_{k,n} + re^{it})\right|\,dt \le \omega_{\phi}(r).
\end{align}

A direct evaluation of the potentials gives that
\[
U^{\nu_k^r}(z)=-\log\max(r,|z-\alpha_{k,n}|),\quad z\in\C,
\]
and
\[
U^{\mu_D}(z) = -\log\max(1,|z|),\quad z\in\C,
\]
cf. \cite{ST}, p. 22. Consider the signed measure $\sigma:=\tau_n^r-\mu_D,\ \sigma(\C)=0.$ One computes (or see \cite{ST}, p. 92) that
\[
d\sigma=-\frac{1}{2\pi}\left(\frac{\partial U^{\sigma}}{\partial n_+} + \frac{\partial U^{\sigma}}{\partial n_-}\right) ds,
\]
where $ds$ is the arclength on $\supp(\sigma)=\{z:|z|=1\}\cup \left( \cup_{k=1}^n \{z:|z-\alpha_{k,n}|=r\}\right)$, and $n_{\pm}$ are the inner and the outer normals. Let $D_R:=\{z:|z|<R\}$ be a disk containing the support of $\phi.$ We now use Green's identity
\[
\iint_G u \Delta v\,dA =  \int_{\partial G} u\,\frac{\partial v}{\partial n}\,ds - \iint_G \nabla u \cdot \nabla v\,dA
\]
with $u=\phi$ and $v=U^{\sigma}$ in each connected component $G$ of $D_R\setminus\supp(\sigma).$ Since $U^{\sigma}$ is harmonic in $G$, we have that $\Delta U^{\sigma}=0$ in $G$. Adding Green's identities for all domains $G$, we obtain that
\begin{align} \label{5.18}
\left|\int\phi\,d\sigma\right| = \frac{1}{2\pi} \left| \iint_{D_R} \nabla \phi \cdot \nabla U^{\sigma} \,dA \right| \le \frac{1}{2\pi} \sqrt{D[\phi]}\,\sqrt{D[U^{\sigma}]},
\end{align}
by the Cauchy-Schwarz inequality. It is known that $D[U^{\sigma}]=2\pi I[\sigma]$ (cf. \cite{La}, Theorem 1.20), where $I[\sigma]=-\iint \log|z-t|\,d\sigma(z)\,d\sigma(t) = \int U^{\sigma}\,d\sigma$ is the energy of $\sigma$. Since $U^{\mu_D}(z) = -\log\max(1,|z|)$, we observe that $\int U^{\mu_D}\,d\mu_D = 0,$ so that
\[
I[\sigma]=\int U^{\tau_n^r}\,d\tau_n^r - 2\int U^{\mu_D}\,d\tau_n^r.
\]
The mean value property of harmonic functions gives that
\begin{align*}
-\int U^{\mu_D}\,d\tau_n^r &= \int\log\max(1,|z|)\, d\tau_n^r(z) \le \frac{1}{n} \left(\sum_{|\alpha_{k,n}|\le 1+r} \log(1+2r) + \sum_{|\alpha_{k,n}|>1+r} \log|\alpha_{k,n}|\right) \\ &\le \log(1+2r) +\frac{1}{n}\log M(P_n) - \frac{1}{n}\log |a_n|.
\end{align*}
We further deduce that
\[
\int U^{\tau_n^r}\,d\tau_n^r = \frac{1}{n^2} \sum_{j,k=1}^n \int U^{\nu_k^r}\,d\nu_j^r \le \frac{1}{n^2} \left(-\sum_{j\neq k} \log|\alpha_{j,n}-\alpha_{k,n}| - n\log{r}\right),
\]
and combine the energy estimates to obtain
\[
I[\sigma] \le \frac{2}{n}\log M(P_n) - \frac{1}{n^2}\log\left|a_n^2 \Delta(P_n)\right| - \frac{1}{n}\log{r} + 4r,
\]
where $ \Delta(P_n)$ is the discriminant of $P_n.$ Using \eqref{5.17}, \eqref{5.18} and the above estimate, we proceed to \eqref{5.16} via the following
\begin{align*}
\left|\int\phi\,d\tau_n - \int\phi\,d\mu\right| &\le \left|\int\phi\,d\tau_n - \int\phi\,d\tau_n^r\right| + \left|\int\phi\,d\tau_n^r - \int\phi\,d\mu\right| \\ &\le \omega_{\phi}(r) + \frac{\sqrt{D[\phi]}\sqrt{D[U^{\sigma}]}}{2\pi} = \omega_{\phi}(r) + \sqrt{\frac{D[\phi]}{2\pi}}\,\sqrt{I[\sigma]}.
\end{align*}
\end{proof}

\begin{proof}[Proof of Theorem \ref{thm3.1}]
We apply Theorem \ref{thm5.2}. Note that $D[\phi]\le 2\pi R^2 A^2,$ as $|\phi_x|\le A$ and $|\phi_y|\le A$ a.e. in $\C.$ Also, it is clear that $\omega_{\phi}(r)\le Ar.$ Since $P_n$ has integer coefficients and simple zeros, we obtain as before that $|\Delta(P_n)|\ge 1$, see \cite{Pra}, p. 24. Combining this with the inequality $|a_n|\ge 1$, we have $|a_n^2\Delta(P_n)|\ge 1$. Hence \eqref{3.2} follows from \eqref{5.16} by letting $r=1/n$, and inserting the above estimates. Note that we also used $\log\max(n,M(P_n)) \ge \log{n} > 4$ for $n\ge 55.$
\end{proof}

\begin{proof}[Proof of Corollary \ref{cor3.2}]
Since $P_n$ has real coefficients, we have that
\[
A_n = \int z\,d\tau_n(z) = \int \Re(z)\,d\tau_n(z).
\]
We now let
\[
\phi(z):=\left\{
           \begin{array}{ll}
             \Re(z), & |z|\le 1; \\
             \Re(z)(1-\log|z|), & 1\le |z|\le e; \\
             0, & |z|\ge e.
           \end{array}
         \right.
\]
An elementary computation shows that $\phi_x$ and $\phi_y$ exist on $\C\setminus S,$ where $S:=\{z:|z|=1 \mbox{ or } |z|=e\}.$ Furthermore, $|\phi_x(z)|\le 1$ and $|\phi_y(z)|\le 1/2$ for $z=x+iy\in\C\setminus S.$ The Mean Value Theorem gives
\[
|\phi(z)-\phi(t)|\le |z-t|\, \sup_{\C\setminus S} \sqrt{\phi_x^2+\phi_y^2} \le \frac{\sqrt{5}}{2}\, |z-t|.
\]
Hence we can use Theorem \ref{thm3.1} with $A=\sqrt{5}/2$ and $R=e.$
\end{proof}

\begin{proof}[Proof of Corollary \ref{cor3.3}]
Note that $\log|P_n(z)| = n \int\log|z-w|\,d\tau_n(w).$ For any $z$ with $|z|=1+1/n,$ we let
\[
\phi(w):=\left\{
           \begin{array}{ll}
             \log|z-w|, & |w|\le 1; \\
             (1-\log|w|)\log|1-\bar{z}w|, & 1\le |w|\le e; \\
             0, & |w|\ge e.
           \end{array}
         \right.
\]
Then $\phi$ is continuous in $\C$, and $\phi_x$ and $\phi_y$ exist on $\C\setminus S,$ where $S:=\{z:|z|=1 \mbox{ or } |z|=e\}.$ We next obtain that $|\phi_x(w)|=O(|z-w|^{-1})$ for $|w|<1,$ and $|\phi_x(w)|=O(|1-\bar{z}w|^{-1})$ for $1<|w|<e.$ Clearly, the same estimates hold for $|\phi_y|.$ Hence
\[
D[\phi] = O\left(\iint_{|w|\le 1} |z-w|^{-2} dA(w)\right) =  O\left(\int_{1/n}^1 r^{-1}dr \right) = O(\log{n}),
\]
and
\[
\omega_\phi(r) \le r \sup_{\C\setminus S} \sqrt{\phi_x^2+\phi_y^2} = r O(n),
\]
as $n\to\infty.$ We let $r=n^{-2}$, and use \eqref{5.16} to obtain
\[
\left|\frac{1}{n}\log|P_n(z)| - \log|z|\right| = O\left(\frac{1}{n}\right) + O(\sqrt{\log{n}}) \left(\frac{2}{n}\log{M} + \frac{2}{n}\log{n} + \frac{4}{n^2}\right)^{1/2}.
\]
Observe that all constants in $O$ terms are absolute. Recall that $|z|=1+1/n$, so that $n\log|z|\to 1$ as $n\to\infty.$ Thus the estimate for $\|P_n\|_D$ follows from the above inequality by the Maximum Principle.
\end{proof}

A close inspection of the proof of Theorem \ref{thm5.2} reveals that it may be easily extended to more general sets. In fact, far more general than those considered below. Define the distance from a point $z\in\C$ to a compact set $E$ by
\[
d_E(z):=\min_{t\in E} |z-t|.
\]
\begin{theorem} \label{thm5.3}
Let $E\subset\C$ be a compact set of capacity 1 that is bounded by finitely many piecewise smooth curves and arcs. Suppose that $\phi:\C\to\R$ is a continuous function with compact support in the plane, and $D[\phi]<\infty.$ If $P_n(z)=a_n \prod_{k=1}^{n} (z-\alpha_{k,n}),\ a_n\neq 0,$ is a polynomial with simple zeros, then for any $r>0$
\begin{align} \label{5.19}
&\left|\int\phi\,d\tau_n - \int\phi\,d\mu_E\right| \le \omega_{\phi}(r) \\ \nonumber &+ \sqrt{\frac{D[\phi]}{2\pi}}\,\left(\frac{2}{n}\log M_E(P_n) - \frac{\log\left|a_n^2 \Delta(P_n)\right|}{n^2} - \frac{\log{r}}{n} + 2\max_{d_E(z)\le 2r} g_E(z,\infty) \right)^{1/2}.
\end{align}
\end{theorem}

\begin{proof}
The proof is very close to that of Theorem \ref{thm5.2}. We sketch it using the same notation, and indicating the necessary changes. Observe that \eqref{5.17} holds without change. Note that $E$ is regular under our assumptions (cf. \cite{Ts}, p. 104), so that $M_E(P_n)=\tilde M_E(P_n)$. We set $g_E(z,\infty) = -U^{\mu_E}(z),\ z\in\C,$ which gives that $g_E(z,\infty) = 0,\ z\in\C\setminus\Omega_E.$

For the signed measure $\sigma:=\tau_n^r-\mu_E,\ \sigma(\C)=0,$ one still has that
\[
d\sigma=-\frac{1}{2\pi}\left(\frac{\partial U^{\sigma}}{\partial n_+} + \frac{\partial U^{\sigma}}{\partial n_-}\right) ds,
\]
where $ds$ is the arclength on $\supp(\sigma)=\{z:z\in\supp(\mu_E)\}\cup \left( \cup_{k=1}^n \{z:|z-\alpha_{k,n}|=r\}\right)$, and $n_{\pm}$ are the inner and the outer normals. This follows from Theorem 1.1 of \cite{Pr2}, see also Example 1.2 there. We use Green's identity to obtain \eqref{5.18} in the same way as in the proof of Theorem \ref{thm5.2}. The energy estimates proceed with the only difference in the following inequality. Since $g_E(z,\infty)$ is harmonic in $\Omega_E$, the mean value property gives that
\begin{align*}
-\int U^{\mu_E}\,d\tau_n^r &= \int g_E(z,\infty)\, d\tau_n^r(z) \le \frac{1}{n} \left(\sum_{d_E(\alpha_{k,n})\le r} \max_{d_E(z)\le 2r} g_E(z,\infty) + \sum_{d_E(\alpha_{k,n})>r} g_E(\alpha_{k,n},\infty)\right) \\ &\le\max_{d_E(z)\le 2r} g_E(z,\infty) +\frac{1}{n}\log M_E(P_n) - \frac{1}{n}\log |a_n|.
\end{align*}
Hence the energy estimates give
\[
I[\sigma] \le \frac{2}{n}\log M_E(P_n) - \frac{1}{n^2}\log\left|a_n^2 \Delta(P_n)\right| - \frac{1}{n}\log{r} + 2\max_{d_E(z)\le 2r} g_E(z,\infty),
\]
and \eqref{5.19} follows by repeating the same argument as in the proof of Theorem \ref{thm5.2}.
\end{proof}

\begin{proof}[Proof of Theorem \ref{thm3.4}]
We deduce \eqref{3.3} from \eqref{5.19}. As in the proof of Theorem \ref{thm3.1}, we obtain that $D[\phi]\le 2\pi R^2 A^2$ and $\omega_{\phi}(r)\le Ar.$ Since $P_n$ has integer coefficients and simple zeros, we also have $|a_n^2\Delta(P_n)|\ge 1$. Recall that the Green function is invariant under translations, so that we may
assume $[a,b]=[-2,2]$. An elementary complex analysis argument gives for $g_{[-2,2]}(z,\infty)=\log|z+\sqrt{z^2-4}|-\log{2}$ that
\begin{align} \label{5.20}
\max_{d_{[-2,2]}(z)\le \eps} g_{[-2,2]}(z,\infty) &= g_{[-2,2]}(2+\eps,\infty) = \log(1+(\eps+\sqrt{4\eps+\eps^2})/2) \\ \nonumber &\le (\eps+\sqrt{4\eps+\eps^2})/2 \le 1.11\sqrt{\eps},\quad 0<\eps\le 0.04.
\end{align}
Now let $r=n^{-2}$, and apply the above estimates in \eqref{5.19} to obtain
\begin{align*}
\left|\int\phi\,d\tau_n - \int\phi\,d\mu_{[a,b]}\right| &\le \frac{A}{n^2} + \sqrt{\frac{2\pi R^2 A^2}{2\pi}}\,\left(\frac{2}{n}\log M_{[a,b]}(P_n) + \frac{2\log{n}}{n} + \frac{2.22\sqrt{2}}{n} \right)^{1/2} \\ &\le A(R\sqrt{5}+1) \sqrt{\frac{\log\max(n,M_{[a,b]}(P_n))}{n}}, \quad n\ge 25.
\end{align*}
Note that we used $r\le 0.04$ for $n\ge 25,$ and $\log\max(n,M_{[a,b]}(P_n)) \ge \log{n} > 2.22\sqrt{2}$ for $n\ge 25.$
\end{proof}

\begin{proof}[Proof of Corollary \ref{cor3.5}]
Consider
\[
\phi(x,y):=\left\{
           \begin{array}{ll}
             x(1-|y|), & a\le x \le b,\ |y|\le 1; \\
             a(1-|y|)(x+1-a), & a-1\le x\le a,\ |y|\le 1; \\
             b(1-|y|)(b+1-x), & b\le x\le b+1,\ |y|\le 1; \\
             0, & \mbox{ otherwise}.
           \end{array}
         \right.
\]
Computing partial derivatives, we see that $|\phi_x|\le \max(|a|,|b|)$ and $|\phi_y|\le \max(|a|,|b|)$ a.e. in $\C.$ Hence $D[\phi] \le 24 \max(|a|^2,|b|^2)$ and $|\phi(z)-\phi(t)|\le \sqrt{2}\max(|a|,|b|)\, |z-t|.$
We use \eqref{5.19} with $r=n^{-2}$ as in the proof of Theorem \ref{thm3.4}, applying \eqref{5.20} and the above estimates:
\begin{align*}
&\left|\int x\,d\tau_n(x) - \int x\,d\mu_{[a,b]}(x)\right| \le \frac{\sqrt{2}\max(|a|,|b|)}{n^2} \\ &+ \sqrt{\frac{24 \max(|a|^2,|b|^2)}{2\pi}}\,\left(\frac{2}{n}\log M_{[a,b]}(P_n) + \frac{2\log{n}}{n} + \frac{2.22\sqrt{2}}{n} \right)^{1/2} \\ &\le (\sqrt{2}+2\sqrt{15/\pi})\max(|a|,|b|) \sqrt{\frac{\log\max(n,M_{[a,b]}(P_n))}{n}}, \quad n\ge 25.
\end{align*}
It remains to observe that $M_{[a,b]}(P_n)\le M$ for $P_n\in\Z_n^s([a,b],M)$, and that
\[
\int x\,d\mu_{[a,b]}(x) = \int_a^b \frac{x\,dx}{\pi\sqrt{(x-a)(b-x)}} = \frac{a+b}{2}.
\]
\end{proof}

\begin{proof}[Proof of Corollary \ref{cor3.6}]
Consider
\[
\phi(x,y):=\left\{
           \begin{array}{ll}
             x^2(1-|y|), & |x|\le 2,\ |y|\le 1; \\
             4(1-|y|)(3-|x|), & 2\le |x|\le 3,\ |y|\le 1; \\
             0, & \mbox{otherwise}.
           \end{array}
         \right.
\]
We find for the partial derivatives that $|\phi_x|\le 4$ and $|\phi_y|\le 4$ a.e. in $\C.$ Hence $D[\phi] \le 384$ and $|\phi(z)-\phi(t)|\le 4\sqrt{2}\, |z-t|.$ We again use \eqref{5.19} with $r=n^{-2}$ as in the proof of Theorem \ref{thm3.4}, applying \eqref{5.20} and the above estimates:
\begin{align*}
\left|\int x^2 \,d\tau_n(x) - \int x^2 \,d\mu_{[-2,2]}(x)\right| &\le \frac{4\sqrt{2}}{n^2} + \sqrt{\frac{384}{2\pi}}\,\left(\frac{2}{n}\log M_{[-2,2]}(P_n) + \frac{2\log{n}}{n} + \frac{2.22\sqrt{2}}{n} \right)^{1/2} \\ &\le 4(\sqrt{2}+2\sqrt{15/\pi}) \sqrt{\frac{\log\max(n,M_{[-2,2]}(P_n))}{n}}, \quad n\ge 25.
\end{align*}
Note that $M_{[-2,2]}(P_n)\le M$ for $P_n\in\Z_n^s([-2,2],M)$, and that
\[
\int x^2\,d\mu_{[-2,2]}(x) = \int_{-2}^2 \frac{x^2\,dx}{\pi\sqrt{4-x^2}} = 2.
\]
\end{proof}

\begin{proof}[Proof of Corollary \ref{cor3.7}]
Consider $z\in\C$ such that $g_{[-2,2]}(z,\infty)=1/n,\ n=\deg(P_n).$ For each $n$, the set of such points is a level curve of the Green function, which is an ellipse enclosing $[-2,2].$ Define
\[
\phi(x,y):=\left\{
           \begin{array}{ll}
             (1-|y|)\log|z-x|, & |x|\le 2,\ |y|\le 1; \\
             (x+3)(1-|y|)\log|z+2|, & -3\le x\le -2,\ |y|\le 1; \\
             (3-x)(1-|y|)\log|z-2|, & 2\le x\le 3,\ |y|\le 1; \\
             0, & \mbox{ otherwise}.
           \end{array}
         \right.
\]
It is clear that $\phi$ is continuous in $\C$, and $\phi_x$ and $\phi_y$ exist a.e. in $\C.$ We have that $|\phi_x(x,y)| \le \max(\log(4+1/n),\log{n})$ for $2\le |x|\le 3,\ |y|\le 1;$ and $|\phi_x(x,y)|\le 1/|z-x|$ for $|x|\le 2,\ |y|\le 1.$ Also, $|\phi_y(x,y)| \le \max(\log(4+1/n),\log{n})$ for $|x|\le 3,\ |y|\le 1.$ Following an argument similar to the proof of Corollary \ref{cor3.3}, we obtain that
$D[\phi] = O(\log{n})$ and $\omega_\phi(r) = r O(n)$ as $n\to \infty$, with absolute constants in $O$ terms. Note that
\[
\int\log|z-x|\,d\tau_n(x) = \frac{1}{n} \log|P_n(z)|
\]
and
\[
\int\log|z-x|\,d\mu_{[-2,2]}(x) = g_{[-2,2]}(z,\infty) = \frac{1}{n}
\]
by \eqref{5.2} and the choice of $z$. We let $r=1/n^2$, and use \eqref{5.19} and \eqref{5.20} as in the proof of Corollary \ref{cor3.6} to obtain
\begin{align*}
\left|\frac{1}{n}\log|P_n(z)| - \frac{1}{n}\right| &= O\left(\frac{1}{n}\right) + O(\sqrt{\log{n}}) \left(\frac{2}{n}\log M_{[-2,2]}(P_n) + \frac{2\log{n}}{n} + \frac{2.22\sqrt{2}}{n} \right)^{1/2}\\ &\le O(\sqrt{\log{n}})\, \sqrt{\frac{\log\max(n,M_{[-2,2]}(P_n))}{n}}, \quad n\ge 25.
\end{align*}
Note that $M_{[-2,2]}(P_n)\le M$ for $P_n\in\Z_n^s([-2,2],M)$. Thus the estimate for $\|P_n\|_{[-2,2]}$ follows from the above inequality by the Maximum Principle.
\end{proof}

\medskip
\noindent{\bf Acknowledgement.} The author would like to thank Al Baernstein for helpful discussions about this paper.


\begin{thebibliography}{99}

\bibitem{AP} J. Aguirre and J. C. Peral,   The trace problem for totally positive algebraic integers. In ``Number theory and polynomials" (Conference proceedings, University of Bristol, 3-7 April 2006, editors James McKee and Chris Smyth), LMS Lecture Notes {\bf 352}, Cambridge, 2008, 1--19.
\bibitem{ABP} J. Aguirre, M. Bilbao, and J. C. Peral,   The trace
of totally positive algebraic integers. Math. Comp. {\bf 75} (2006),
385--393.
\bibitem{AM} F. Amoroso and M. Mignotte,   On the distribution of the roots of polynomials. Ann. Inst. Fourier (Grenoble) {\bf 46} (1996), 1275--1291.
\bibitem{AB} V. V. Andrievskii and H.-P. Blatt, Discrepancy of
signed measures and polynomial approximation. Springer-Verlag, New
York, 2002.
\bibitem{BLP} A. Baernstein II, R. S. Laugesen, and I. E. Pritsker,   Moment inequalities for equilibrium measures in the plane. Pure Appl. Math. Q. (to appear)
\bibitem{Bi} Y. Bilu,   Limit distribution of small points on
algebraic tori. Duke Math. J. {\bf 89} (1997), 465--476.
\bibitem{Bo} E. Bombieri,   Subvarieties of linear tori and the
unit equation: A survey. In ``Analytic number theory," ed. by Y.
Motohashi, LMS Lecture Notes {\bf 247} (1997), Cambridge Univ. Press,
Cambridge, pp. 1--20.
\bibitem{Bor1} P. Borwein, Computational excursions in analysis and
number theory. Springer-Verlag, New York, 2002.
\bibitem{BE} P. Borwein and T. Erd\'elyi,   The integer Chebyshev
problem. Math. Comp. {\bf 214} (1996), 661--681.
\bibitem{Di} A. Dinghas,   Sur un th\'eor\`eme de Schur concernant
les racines d'une classe des \'equations alg\'ebriques. Norske Vid.
Selsk. Forh., Trondheim {\bf 25} (1952), 17--20.
\bibitem{DuSm} A. Dubickas and C. J. Smyth,   The Lehmer constants of an annulus. J. Th\'eor. Nombres Bordeaux {\bf 13} (2001), 413--420.
\bibitem{DuSm2} A. Dubickas and C. J. Smyth,   Two variations of a theorem of Kronecker. Expo. Math. {\bf 23} (2005), 289--294.
\bibitem{ET} P. Erd\H{o}s and P. Tur\'an,   On the distribution
of roots of polynomials. Ann. Math. {\bf 51} (1950), 105--119.
\bibitem{ET2} P. Erd\H{o}s and P. Tur\'an,   On the uniformly-dense distribution of certain sequences of points. Ann. Math. {\bf 41} (1940), 162--173.
\bibitem{FR} C. Favre and J. Rivera-Letelier,   Equidistribution quantitative des points de petite hauteur sur la droite projective. Math. Ann. {\bf 335} (2006), 311--361; Corrigendum in Math. Ann. {\bf 339} (2007), 799--801.
\bibitem{Fe} M. Fekete,   \"{U}ber die Verteilung der Wurzeln bei
gewissen algebraischen Gleichungen mit ganzzahligen Koeffizienten.
Math. Zeit. {\bf 17} (1923), 228--249.
\bibitem{FRS} V. Flammang, G. Rhin, and C. J. Smyth,   The integer
transfinite diameter of intervals and totally real algebraic
integers. J. Theor. Nombres-Bordeaux {\bf 9} (1997), 137--168.
\bibitem{Ga} T. Ganelius,   Sequences of analytic functions and their zeros. Ark. Mat. {\bf 3} (1953), 1--50.
\bibitem{HLP} G. H. Hardy, J. E. Littlewood and G. P\'olya,
Inequalities. Cambridge Univ. Press, London, 1952.
\bibitem{Hu} J. Hunter,   A generalization of the inequality of the
arithmetic-geometric means. Proc. Glasgow Math. Assoc. {\bf 2} (1956),
149--158.
\bibitem{Hus} J. Huesing,   Estimates for the discrepancy of a signed measure using its energy norm. J. Approx. Theory {\bf 109} (2001) 1--29.
\bibitem{Je} R. Jentzsch,   Untersuchungen zur Theorie der Folgen
analytischer Funktionen. Acta Math. {\bf 41} (1917), 219--270.
\bibitem{Kl} W. Kleiner,   Une condition de Dini-Lipschitz dans la th\'eorie du potentiel. Ann. Polon. Math. {\bf 14} (1964), 117--130.
\bibitem{Kr} L. Kronecker,   Zwei S\"{a}tze \"{u}ber
Gleichungen mit ganzzahligen Co\"efficienten. J. reine angew. Math.
{\bf 53} (1857), 173--175.
\bibitem{La} N. S. Landkof, Foundations of modern potential theory.
Springer-Verlag, New York - Heidelberg, 1972.
\bibitem{Lan} S. Lang, Fundamentals of diophantine geometry. Springer-Verlag, New York, 1983.
\bibitem{La1} M. Langevin,   M\'ethode de Fekete-Szeg\"o et
probl\`eme de Lehmer. C. R. Acad. Sci. Paris, S\'er. I Math.
{\bf 301} (1985), 463--466.
\bibitem{La2} M. Langevin,   Minorations de la maison et de la mesure de
Mahler de certains entiers alg\'ebriques. C. R. Acad. Sci.
Paris, S\'er. I Math. {\bf 303} (1986), 523--526.
\bibitem{La3} M. Langevin,   Calculs explicites de constantes de
Lehmer. Groupe de travail en th\'eorie analytique et \'el\'ementaire
des nombres, 1986--1987, 52--68, Publ. Math. Orsay, 88-01, Univ.
Paris XI, Orsay, 1988.
\bibitem{Mi} M. Mignotte,   Sur un th\'eor\`eme de M. Langevin.
Acta Arith. {\bf 54} (1989), 81--86.
\bibitem{Pe} C. Petsche,   A quantitative version of Bilu's equidistribution theorem. Int. J. Number Theory {\bf 1} (2005),  281--291.
\bibitem{Pra} V. V. Prasolov, Polynomials. Springer, Berlin, 2004.
\bibitem{Pr1} I. E. Pritsker,   Small polynomials with integer coefficients. J. Anal. Math. {\bf 96} (2005), 151--190.
\bibitem{Pr2} I. E. Pritsker,   How to find a measure from its potential. Comput. Methods Funct. Theory {\bf 8} (2008), 597--614.
\bibitem{Pr3} I. E. Pritsker,   Means of algebraic numbers in the unit disk. C. R. Acad. Sci. Paris, S\'er. I {\bf 347} (2009), 119--122.
\bibitem{Pr4} I. E. Pritsker,  Equidistribution of points via energy. Ark. Mat. (to appear)
\bibitem{Ra} T. Ransford, Potential theory in the complex plane.
Cambridge University Press, Cambridge, 1995.
\bibitem{RhSm} G. Rhin and C. J. Smyth,   On the absolute Mahler
measure of polynomials having all zeros in a sector. Math.
Comp. {\bf 65} (1995), 295--304.
\bibitem{Ri} T. J. Rivlin, Chebyshev polynomials. John Wiley
\& Sons, New York, 1990.
\bibitem{Ro1} R. M. Robinson, Intervals containing
infinitely many sets of conjugate algebraic integers. In
``Studies in Mathematical Analysis and Related Topics: Essays in
Honor of George P\'olya," Stanford, 1962, pp. 305--315.
\bibitem{Ro2} R. M. Robinson, Conjugate algebraic integers in real
point sets. Math. Zeit. {\bf 84} (1964), 415-427.
\bibitem{Ro3} R. M. Robinson, Conjugate algebraic integers on a
circle. Math. Zeit. {\bf 110} (1969), 41-51.
\bibitem{Ru} R. Rumely,   On Bilu's equidistribution theorem. In
``Spectral problems in geometry and arithmetic (Iowa City, IA,
1997)," Contemp. Math. {\bf 237}, Amer. Math. Soc., Providence, RI, 1999,
pp. 159--166.
\bibitem{ST} E. B. Saff and V. Totik, Logarithmic potentials with
external fields. Springer-Verlag, Berlin, 1997.
\bibitem{Sc} A. Schinzel,   On the product of the conjugates
outside the unit circle of an algebraic number. Acta Arith.
{\bf 24} (1973), 385--399; Addendum: Acta Arith. {\bf 26} (1974/75), 329--331.
\bibitem{Sch} I. Schur,   \"{U}ber die Verteilung der Wurzeln bei
gewissen algebraischen Gleichungen mit ganzzahligen Koeffizienten.
Math. Zeit. {\bf 1} (1918), 377--402.
\bibitem{Si} C. L. Siegel,   The trace of totally positive and real
algebraic integers. Ann. Math. {\bf 46} (1945), 302--312.
\bibitem{Sj} P. Sj\"ogren,   Estimates of mass distributions from their potentials and energies. Ark. Mat. {\bf 10} (1972), 59--77.
\bibitem{Sm1} C. J. Smyth,   Totally positive algebraic integers of
small trace. Ann. Inst. Fourier Grenoble {\bf 34} (1984), 1--28.
\bibitem{Sm2} C. J. Smyth,   The mean values of totally real algebraic integers. Math. Comp. {\bf 42} (1984), 663--681.
\bibitem{Sm3} C. J. Smyth,   An inequality for polynomials.
CRM Proceedings and Lecture Notes {\bf 19} (1999), 315--321.
\bibitem{Sm4} C. J. Smyth,   The Mahler measure of algebraic numbers: a survey. In ``Number theory and polynomials" (Conference proceedings, University of Bristol, 3-7 April 2006, editors James McKee and Chris Smyth), LMS Lecture Notes {\bf 352}, Cambridge, 2008, 322--349.
\bibitem{Sz} G. Szeg\H{o},  \"{U}ber die Nullstellen von Polynomen,
die in einem Kreis gleichm\"{a}ssig konvergieren. Sitzungsber. Ber.
Math. Ges. {\bf 21} (1922), 59--64.
\bibitem{Ts} M. Tsuji, Potential theory in modern function theory.
Chelsea Publ. Co., New York, 1975.

\end{thebibliography}
\end{document}